\documentclass[11pt]{article}

\usepackage{datetime}

\usepackage{color}

\usepackage{eucal}

\usepackage[all]{xy}

\usepackage{tikz}

\usepackage{amssymb, amsmath, amsthm, eucal}

\usepackage{mathptmx}
\usepackage[scaled=.90]{helvet}
\usepackage{courier}

\usepackage[
color=gray!40,
bordercolor=black,
textwidth=.8in]
{todonotes}
\makeatletter \providecommand\@dotsep{5}
\makeatother

\usepackage[pdftex]{hyperref}

\hypersetup{
  colorlinks = true,
  linkcolor = black,
  urlcolor =  blue,
  citecolor = black,
}

\newtheorem{theorem}{Theorem}
\newtheorem{proposition}[theorem]{Proposition}

\newtheorem{lemma}[theorem]{Lemma}

\theoremstyle{definition}
\newtheorem{definition}[theorem]{Definition}
\newtheorem{example}[theorem]{Example}

\newtheorem{remark}[theorem]{Remark}

\numberwithin{theorem}{section}
\numberwithin{equation}{section}
\numberwithin{figure}{section}

\title{
A non-trivial ghost kernel for complex projective spaces with symmetries}

\author{Markus Szymik}

\newdateformat{mydate}{\monthname~\twodigit{\THEYEAR}}
\date{\mydate\today}

\begin{document}

\maketitle

\renewcommand{\abstractname}{}

\begin{abstract}
\noindent 
It is shown that the ghost kernel for certain equivariant stable cohomotopy groups of projective spaces is non-trivial. The proof is based on the Borel cohomology Adams spectral sequence and the calculations with the Steenrod algebra afforded by it.

\vspace{\baselineskip}

\vspace{\baselineskip}
\end{abstract}


\newcommand{\Hom}[3]{\operatorname{Hom}_{#1}^{#2}(#3)}
\newcommand{\Ext}[3]{\operatorname{Ext}_{#1}^{#2}(#3)}


\section{Introduction}

Equivariant stable homotopy theory deals with the algebraic topology of group actions. Symmetries are often useful even when studying problems that do not involve group actions in the first place. The latest successful implementation of this idea appears in the work of Hill, Hopkins, and Ravenel on the non-existence of elements of Kervaire invariant one, see~\cite{HHR:pre} and the surveys~\cite{HHR:1}, ~\cite{HHR:2},~\cite{HHR:3}, and~\cite{Miller}.

In equivariant stable homotopy theory there are many results which compare the equivariant stable homotopy category to the usual, non-equivariant, stable homotopy category. See~\cite{LMS} and~\cite{Mayetal} for background. For example, it is well-known that~$G$-equivalences~$f$ can be detected by their~$H$-fixed points~$f^H$ for the various subgroups~\hbox{$H\leqslant G$}.  In contrast, the analogous statement is known to be false for the class of~$G$-null maps, and this will be amplified here.

For easy book-keeping, let us fix a prime~$p$, and let~$G$ be a finite group of order~$p$. These groups have precisely two subgroups, and it will turn out that even this simple case is interesting enough for a start. The group~$[X,Y]^G$ of equivariant stable homotopy classes between~(pointed)~\hbox{$G$-CW}-com\-plexes~$X$ and~$Y$ can be studied by means of the {\it ghost map}
\begin{equation*}
  [X,Y]^G
  \longrightarrow 
  \mathrm{H}^0(G;[X,Y])
  \oplus
  [X^G,Y^G],
\end{equation*}
which sends a stable~$G$-map~$f$ to the pair~$(f,f^G)$. Clearly, the kernel of a ghost map is a place to look for genuinely equivariant phenomena. See~\cite{Christensen} for a conceptual framework for related matters.
     The following examples show that ghost kernels can be non-trivial, so that, contrary to the case of equivalences, maps which are essential cannot necessarily be detected by their ghosts.

\begin{example}
	The target of the ghost map for~$[\mathrm{S}^1,\mathrm{S}^0]^G$ is~$\mathbb{Z}/2\oplus\mathbb{Z}/2$ if both spheres have the trivial~$G$-action. However, the splitting theorem shows that the group~$[\mathrm{S}^1,\mathrm{S}^0]^G$ has an extra summand~$[\mathrm{S}^1,\mathrm{B}G]\cong G$ for the present choice of~$G$.
\end{example}

\begin{example}
 The `boundary map' of the cofibration sequence~$\mathrm{E}G_+\to \mathrm{S}^0\to\widetilde{\mathrm{E}}G$ is in the ghost kernel of~$[\widetilde{\mathrm{E}}G,\Sigma\mathrm{E}G_+]^G$: the target of the ghost map is zero because both spaces~$\widetilde{\mathrm{E}}G$ and~$(\mathrm{E}G_+)^G$ are contractible. However, if that map were zero, this would split~$\mathrm{S}^0$ into~$\mathrm{E}G_+\vee\widetilde{\mathrm{E}}G$, but there is no non-trivial idempotent in the Burnside ring~$[\mathrm{S}^0,\mathrm{S}^0]^G$.
\end{example}

The main aim of this text is to display a new family of elements in the ghost kernel for another  naturally occurring situation: the equivariant stable cohomotopy of projective spaces. Apart from the intrinsic interest in these spaces~\cite{Greenlees+Williams} and maps like this as explained above, this is also used in~\cite{Szymik:Galois} to show that the equivariant Bauer--Furuta invariants of Galois covers of smooth 4-manifolds are not determined by non-equivariant data. (See also~\cite{Nakamura:1},~\cite{Nakamura:2},~\cite{Nakamura:3}, and~\cite{Nakamura:4}.) It has been these applications which originally led to the work presented here. 

Let us now investigate in more detail the~$G$-spaces~$X$ and~$Y$ which will be considered here, and why they are of interest also from the point of view of pure equivariant homotopy theory. If~$V$ is a complex~$G$-representation, then~$\mathbb{C}\mathrm{P}(V)_+$ denotes the complex projective space of~$V$ with the induced~$G$-action, and with a disjoint~$G$-fixed base point added. The other fixed points of~$\mathbb{C}\mathrm{P}(V)_+$ are the~$G$-invariant lines~$L$ in~$V$. In the language of representation theory, the fixed point sets are the projective spaces of the isotypical components of~$V$ which correspond to~1-dimensional irreducible representations of~$G$. If~$G$ is abelian, as in our case, all irreducible representations are~1-dimensional. This shows that the fixed point sets all have the same dimension if and only if~$V$ is a multiple of the regular representation~$\mathbb{C}G$. That singles out this case, and we shall therefore assume it from now on. Also, the~$G$-spaces~$\mathbb{C}\mathrm{P}(V)$ appear as skeleta of the classifying space for~$G$-line bundles, underlining again their general importance.

At a fixed point~$L$ of~$\mathbb{C}\mathrm{P}(V)$, the tangent~$G$-representation  is $\Hom{}{}{L,V/L}$.~(For typographical reasons, a notation such as~$V/L$ may refer to the~(orthogonal) complement of~$L$ in~$V$ in the following.) If~$V$ is a multiple of the regular representation, these tangent representations are all isomorphic, namely to~$W=V/\,\mathbb{C}$. This is yet another fact that makes the present choice of~$V$ special, and it also shows that~$\mathrm{S}^W$ is a natural target in this situation. The corresponding collapse maps~\hbox{$\mathbb{C}\mathrm{P}(V)_+\to\mathrm{S}^W$}
to the one-point-compactification~$\mathrm{S}^W$ of~$W$ generate the group
\begin{displaymath}
	[\mathbb{C}\mathrm{P}(V)_+,\mathrm{S}^W]^G,
\end{displaymath}
see~\cite{Szymik:Hopf}, where it is shown that the ghost map is injective in that case. As it will turn out, we can change this if we replace~$W$ by~$W/\,\mathbb{R}G$, and this is situation studied here. Note that~$\mathbb{R}G$ can be embedded in~$a\mathbb{C}G/\,\mathbb{C}$ if and only if~$a\geqslant2$. 

The preceding discussion motivates the following notation which will be used throughout the entire text. Choose an integer~$a\geqslant2$ and let
\begin{equation}\label{eq:notation}
	V_a=a\mathbb{C}G\quad\text{and}\quad W_a=a\mathbb{C}G/(\mathbb{C}\oplus \mathbb{R} G),
\end{equation}
considered as a complex and real~$G$-representation, respectively. Then the main result proved in this text is the following.

\begin{theorem}\label{theorem:introduction}
  For all integers $a\geqslant2$ and all primes~$p\geqslant5$, the ghost kernel for
  \[
  [\mathbb{C}\mathrm{P}(V_a)_+,\mathrm{S}^{W_a}]^G
  \]
  is non-trivial.
\end{theorem}

It is easily seen that the target of the ghost map is zero in this case, so that it suffices to be shown that the group is non-zero, which is everything but obvious. In fact, slightly more will be proven, see~Theorem~\ref{theorem:structure}: For all~$p\geqslant 3$ the~$p$-power torsion of this group is an elementary abelian~$p$-group of rank~$r$ for some integer~$r$ that satisfies~$1\leqslant r\leqslant (p+1)/2$. While this not only shows the non-triviality of the group, it also gives an upper bound on its structure.~I have reasons to conjecture that the group turns out to have~$p$-rank~$1$ in all cases; see Remark~\ref{rem:conj} at the end of Section~\ref{sec:conclusion}. Furthermore, it is tempting to relate the groups for various~$a\geqslant 2$, and this problem is discussed in the final Section~\ref{sec:family}, see~Remarks~\ref{rem:1},~\ref{rem:2}, and~\ref{rem:3}.

The proof of Theorem~\ref{theorem:structure} relies on the Adams spectral sequence based on Borel cohomology~\cite{Greenlees:Free} and on the computations done with it in~\cite{Szymik:ESS}. In Sections~\ref{sec:borel} and~\ref{sec:projective}, we will recall the necessary facts about that spectral sequence and about the Borel cohomology of projective spaces, respectively. The filtration of these by projective subspaces will later (in Section~\ref{sec:obstructions}) be used to feed in the computations of~\cite{Szymik:ESS}. Before that, in Sections~\ref{sec:toptoalg},~\ref{sec:a=2}, and~\ref{sec:a>2} a new method is introduced to algebraically calculate the groups of homomorphisms of modules over the Steenrod algebra in the relevant cases. The results obtained there are the other, new ingredient which is needed as an input for the computation. Section~\ref{sec:conclusion} combines all this to give a proof of the main theorem. As already mentioned above, there and in the concluding Section~\ref{sec:family} we also discuss some related open problems.


\section{Borel cohomology and its Adams spectral sequence}\label{sec:borel}

Let~$p$ be an odd prime number, and let~$G$ be the cyclic group of
order~$p$. The notation~$\mathrm{H}^*$ will be used for (reduced) ordinary
cohomology with coefficients in the field~$\mathbb{F}$ with~$p$ ele\-ments.
For~a finite~(pointed)~$G$-CW-complex~$X$, the Borel co\-ho\-mo\-logy is defined
as
\begin{displaymath} 
  b^*X=\mathrm{H}^*(\mathrm{E}G_+\wedge_GX).
\end{displaymath}
Therefore, the coefficient ring~$b^*=b^*\mathrm{S}^0=\mathrm{H}^*(\mathrm{B}G_+)$ is the mod~$p$
cohomology ring of the group. Since~$p$ is odd, this is the tensor
product of an exterior algebra on a generator~$\sigma$ in degree~$1$
and~a polynomial algebra on~a generator~$\tau$ in degree~$2$.

If~$X$ and~$Y$ are finite~$G$-CW-complexes, the Borel cohomology Adams
spectral sequence takes the form
\begin{displaymath}
  E_2^{s,t}=\Ext{b^*b}{s,t}{b^*Y,b^*X}
  \Longrightarrow ([X,Y]^G_{t-s})^{\wedge}_p.
\end{displaymath}
The convergence to the indicated target has been established by Greenlees,
see~\cite{Greenlees:Free}.

As~a vector space, the algebra~$b^*b$ is the tensor
product~$A^*\otimes b^*$, where~$A^*$ is the mod~$p$ Steenrod algebra.
The algebra~$A^*$ is generated by the Bockstein element~$\beta$ in
degree~$1$, and the Steenrod powers~$P^i$ for~$i\geqslant1$ in
degree~\hbox{$2i(p-1)$}. By convention,~$P^0$ is the unit of the
Steenrod algebra. Often the total power operation
\begin{displaymath}
  P=\sum_{i=0}^{\infty}P^i
\end{displaymath} 
will be used, which acts multiplicatively on cohomology algebras.  As
an example, the~$A^*$-action on the coefficient ring~$b^*=b^*\mathrm{S}^0$ is
given by
\begin{displaymath}
 \beta(\sigma)=\tau,\quad
 \beta(\tau)=0,\quad
 P(\sigma)=\sigma,\quad
 \hbox{and}\quad
 P(\tau)=\tau+\tau^p.
\end{displaymath}
The multiplication in~$b^*b=A^*\otimes b^*$ is~the twisted product,
the twist being given by the~$A^*$-action on~$b^*$.


\section{Projective spaces and their Borel cohomology}\label{sec:projective}

Let~$V$ by a complex~$G$-re\-pre\-senta\-tion. In order to describe the
Borel co\-ho\-mo\-lo\-gy of the projective spaces~$\mathbb{C}\mathrm{P}(V)_+$,~a few basic facts about
Chern classes of representations~\cite{Atiyah:Completion} will have to
be recalled.

By definition, the Borel cohomology of~$\mathbb{C}\mathrm{P}(V)_+$ is the ordinary cohomology of the space~$\mathrm{E}G_+\wedge_G\mathbb{C}\mathrm{P}(V)_+=\left(\mathrm{E}G\times_G\mathbb{C}\mathrm{P}(V)\right)_+$. The space~\hbox{$\mathrm{E}G\times_G\mathbb{C}\mathrm{P}(V)$} is nothing but the projective bundle associated to the vector
bundle~$\mathrm{E}G\times_G V$ over~$\mathrm{B}G$. There is a tautological line bundle
over the space~\hbox{$\mathrm{E}G\times_G\mathbb{C}\mathrm{P}(V)$}. Write~$\xi_V$ for its
first Chern class. The same symbol will be used for the reduction
modulo~$p$. Let~$n$ be the dimension of~$V$. It follows from the
Leray-Hirsch theorem that the~$b^*$-module~$b^*\mathbb{C}\mathrm{P}(V)_+$ is free
of rank~$n$ with basis~$1,\xi_V,\xi_V^2,\dots,\xi_V^{n-1}$. The
relation
\begin{equation}\label{Chern_sum}
  \sum_{j=0}^n(-1)^jc_j(V)\xi_V^{n-j}=0
\end{equation}
can be used as the definition of the Chern classes~$c_j(V)$ of the
vector bundle~$\mathrm{E}G\times_G V$ over~$\mathrm{B}G$. If~$V=V_1\oplus V_2$ is a direct
sum, then the total Chern class~$c(V)=\sum_{j=0}^nc_j(V)$
equals the product of the total Chern classes of the
summands:~\hbox{$c(V_1\oplus V_2)=c(V_1)\cdot c(V_2)$}.

Since~$G$ is cyclic, the representations are easy to describe. Given
an integer~$\alpha$, let~$\mathbb{C}(\alpha)$ be the~$G$-representation
where a chosen generator of~$G$ acts as multiplication
by~$\exp(2\pi\mathrm{i}\alpha/p)$. We may define~$\tau$ to be the first Chern class of the~$G$-representation~$\mathbb{C}(1)$. 
Since multiplication of irreducible~1-dimensional representations corresponds to addition
in cohomology, the first Chern class of the
representation~$\mathbb{C}(\alpha)$ of~$G$ is~$\alpha\tau$. If
\begin{displaymath}
V=\bigoplus_{j=1}^n\mathbb{C}(\alpha_j),   
\end{displaymath}
consider the polynomial
\begin{displaymath}
  f(V)=\prod_{j=1}^n(x-\alpha_j\tau)
\end{displaymath}
in~$b^*[x]$. For example, if~$V$ is the complex regular
representation~$\mathbb{C}G$,
\begin{equation}\label{r}
  f(\mathbb{C}G)=\prod_{\alpha=1}^p(x-\alpha\tau)=
  x^p-\tau^{p-1}x=x(x^{p-1}-\tau^{p-1}).
\end{equation}
This polynomial will become prominent in the following Section~\ref{sec:toptoalg}. By~(\ref{Chern_sum}), the map from~$b^*[x]$ to~$b^*\mathbb{C}\mathrm{P}(V)_+$ which sends~$x$ to~$\xi_V$ induces an isomorphism
\begin{equation}\label{modell}
        b^*[x]/f(V)\cong b^*\mathbb{C}\mathrm{P}(V)_+.
\end{equation} 

The structure of~$b^*\mathbb{C}\mathrm{P}(V)_+$ as a~$b^*$-module is clear by the
Leray-Hirsch theorem. As for the action of the Steenrod algebra~$A^*$,
it suffices to study that on the generators of~$b^*\mathbb{C}\mathrm{P}(V)_+$ as
a~$b^*$-module and, by multiplicativity, on~$\xi_V$. But this element
has degree~$2$, so the action of the total Steenrod power~$P$ on it is
clear:~$P(\xi_V)=\xi_V+\xi_V^p$. Since~$\xi_V$ is an integral
class,~$\beta(\xi_V)=0$, so~$\beta$ acts trivially on all elements of
even degree.

If~$U\subseteq V$ is a subrepresentation, the inclusion of~$\mathbb{C}
\mathrm{P}(U)_+$ into~$\mathbb{C}\mathrm{P}(V)_+$ induces a surjection in Borel cohomology. In the following, the notation~$V/U$ will often denote an (orthogonal) complement of~$U$ in~$V$, for typographical reasons. If~$V$ contains a~$G$-line~$L$, the cofibre sequence
\begin{equation}
  \label{CP(V) cofib}
  \mathbb{C}\mathrm{P}(V/L)_+\longrightarrow\mathbb{C}\mathrm{P}(V)_+\longrightarrow \mathrm{S}^{\Hom{}{}{L,V/L}}
\end{equation}
induces a short exact sequence
\begin{displaymath}
  0
  \longleftarrow
  b^*\mathbb{C}\mathrm{P}(V/L)_+
  \longleftarrow
  b^*\mathbb{C}\mathrm{P}(V/L\oplus L)_+
  \longleftarrow
  b^*\mathrm{S}^{\Hom{}{}{L,V/L}}
  \longleftarrow
  0
\end{displaymath}
of~$b^*b$-modules, which in turn induces long exact sequences
\begin{equation}\label{typical_les}
  \dots\longleftarrow
  \Ext{b^*b}{s+1,t}{b^*Y,b^*\mathrm{S}^{\Hom{}{}{L,V/L}}}
  \longleftarrow
  \Ext{b^*b}{s,t}{b^*Y,b^*\mathbb{C}\mathrm{P}(V/L)_+}
  \longleftarrow\dots
\end{equation}
for any~$Y$. These will be used frequently later on. The long exact sequences~(\ref{typical_les}) converge to the long
exact sequences
\begin{displaymath}
  \dots\longleftarrow
  [\Sigma^{-1}\mathrm{S}^{\Hom{}{}{L,V/L}},Y]^G_{t-s}
  \longleftarrow
  [\mathbb{C}\mathrm{P}(V/L)_+,Y]^G_{t-s}
  \longleftarrow\dots
\end{displaymath}
induced by the cofibre sequence~(\ref{CP(V) cofib}).


\section{From topology to algebra}\label{sec:toptoalg}

One of our main ingredients for the later calculations with the Adams spectral sequence is the vector space
$\Hom{b^*b}{1}{b^*\mathrm{S}^{W_a},b^*\mathbb{C}\mathrm{P}(V_a)_+}$. The purpose of this section is to prove~Proposition~\ref{translation}, which establishes an isomorphism between that vector space and another one which is defined purely in terms of polynomial algebra.

\begin{definition}\label{def:Mbar}
Let~$\overline{M}_a$ be the vector space of all elements~$\mu$ in~$b^*\mathbb{C}
  \mathrm{P}(V_a)_+$ of degree~$(2a-1)p-3$ for which the equation
  \begin{equation}\label{formula for P action}
    P(\mu)=(1+\tau^{p-1})^{(2a-1)\frac{p-1}{2}}\mu
  \end{equation}
  describes the action of the total Steenrod operation.
\end{definition}

\begin{lemma}\label{lem:1}
  Evaluation on a generator of~$b^*\mathrm{S}^{W_a}$ gives an isomorphism
  bet\-ween the vector space $\Hom{b^*b}{1}{b^*\mathrm{S}^{W_a},b^*\mathbb{C}
    \mathrm{P}(V_a)_+}$ and~$\overline{M}_a$.
\end{lemma}

\begin{proof}
  First let us translate the grading into a suspension, so that we
  then have to deal with the
  group $\Hom{b^*b}{}{\Sigma^{-1}b^*\mathrm{S}^{W_a},b^*\mathbb{C}\mathrm{P}(V_a)_+}$. 
  A~$b^*b$-linear map from~$\Sigma^{-1}b^*\mathrm{S}^{W_a}$ into~$b^*\mathbb{C}
  \mathrm{P}(V_a)_+$ is just a~$b^*$-linear map which is also~$A^*$-linear.
  A~$b^*$-linear map like that is the same as an element~$\mu$
  in~$b^*\mathbb{C}\mathrm{P}(V_a)_+$ of degree one less than the degree of the
  generator of~$b^*\mathrm{S}^{W_a}$.  This is to say that~$\mu$ is to have
  degree~$(2a-1)p-3$. In other symbols:
  \begin{displaymath}
    \Hom{b^*}{1}{\Sigma^{-1}b^*\mathrm{S}^{W_a},b^*\mathbb{C}\mathrm{P}(V_a)_+} \cong
    b^{(2a-1)p-3}\mathbb{C}\mathrm{P}(V_a)_+.
  \end{displaymath}
  Such an element~$\mu$ in~$b^*\mathbb{C}\mathrm{P}(V_a)_+$ corresponds to
  an~$A^*$-linear map if and only if the Steenrod algebra acts
  on~$\mu$ as on the generator of~$b^*\mathrm{S}^{W_a}$. In order to proceed,
  one has to know the~$A^*$-action on the Borel cohomology
  of~$\mathrm{S}^{W_a}$. The (real) dimension of~$W_a$ is~$(2a-1)p-2$ and the
  (real) dimension of its fixed point set is~\hbox{$(2a-1)-2$}. The
  difference is~$(2a-1)(p-1)$. Thus~$A^*$ acts on a generator
  of~$b^*\mathrm{S}^{W_a}$ as on~$\tau^{(2a-1)\frac{p-1}{2}}$ in~$b^*$.  This
  means that~$\beta$ acts trivially, and~$P$ acts by multiplication
  with
  \begin{displaymath}
    (1+\tau^{p-1})^{(2a-1)\frac{p-1}{2}}.
  \end{displaymath} Note that this element is not homogeneous, since~$P$ is not
  homogeneous. This can now be compared to the action of the Steenrod
  algebra on~$\mu$. Since the degree of~$\mu$ is even,~$\beta$ acts
  trivially. Thus, the only condition is on the power operations. One has
  to require that
  \begin{displaymath}
    P(\mu)=(1+\tau^{p-1})^{(2a-1)\frac{p-1}{2}}\mu
  \end{displaymath} 
  in~$b^*\mathbb{C}\mathrm{P}(V_{a-1})_+$ for~$\mu$ to represent a~$b^*b$-linear map.  
\end{proof}

Motivated by the lemma, let us write
\begin{equation}\label{eq:epsilon}
  \epsilon_a=(2a-1)\frac{p-1}2=(p-1)a-\frac{p-1}{2}
\end{equation}
and 
\begin{equation}\label{eq:h}
  h_a=(1+\tau^{p-1})^{\epsilon_a}
\end{equation}
for short. Thus, equation~(\ref{formula for P action}) now
reads~$P(\mu)=h_a\mu$.
        
A comment on the degrees might be appropriate. Until now, all the
degrees have come from the usual topologist's grading of cohomology
groups. Since our problem will eventually be reduced to polynomial
algebra, it will be convenient to use algebraic degrees. Then the
elements~$x$ and~$\tau$ of topological degree~$2$ will have algebraic
degree~$1$. The algebraic degree of~$\mu$ will be written
\begin{equation}\label{eq:delta}
  \delta_a=\frac{1}{2}((2a-1)p-3)=pa-\frac{p+3}{2}.
\end{equation}
From now on, the degrees used will be algebraic unless stated otherwise.

Recall from (\ref{modell}) that one has an isomorphism~$b^*\mathbb{C}
\mathrm{P}(V)_+\cong b^*[x]/f(V)$ for any~$G$-representation~$V$. This will be
used to identify the two rings. In par\-ticu\-lar, we have~\hbox{$b^*\mathbb{C}
\mathrm{P}(V_a)_+=b^*[x]/f(V_a)$}. If we set~$r=f(\mathbb{C} G)$, then
\begin{equation}\label{f}
  r=\prod_{\lambda\in\mathbb{F}}(x-\lambda\tau)=x^p-x\tau^{p-1}=x(x^{p-1}-\tau^{p-1})
\end{equation}
as in (\ref{r}), and~$f(V_a)=f(\mathbb{C} G)^a=r^a$. Note that the algebraic
degree of~$r^a$ is~$ap$.

\begin{definition}\label{def:M}
Let~$M_a$ be the vector space of polynomials~$m$ in the
subring~$\mathbb{F}[\tau,x]$ of~$b^*[x]$ such that the algebraic degree
of~$m$ is~$\delta_a$ and~$r^a$ divides~\hbox{$P(m)-h_am$}.
\end{definition}

\begin{lemma}\label{lem:2}
  There is an isomorphism between~$M_a$ and~$\overline{M}_a$ which is the identity on representatives.
\end{lemma}

\begin{proof}
	With the notation already established, we see that~$\overline{M}_a$ is the vector space of all elements~$\mu$ in~$b^{2\delta_a}\mathbb{C}\mathrm{P}(V_a)_+$ for which the
  equation~$P(\mu)=h_a\mu$ holds. 
  
  First note that the algebraic degree of~$r^a$, which is~$ap$, is larger than the algebraic degree~$\delta_a$ of the~$\mu$. This shows that there are as many elements~$\mu$ of
  that degree in~$b^*\mathbb{C}\mathrm{P}(V_a)_+$ as in~$\mathbb{F}[\tau,x]$. In other
  words, the map
  \begin{displaymath}
    \mathbb{F}[\tau,x]\subset b^*[x]\longrightarrow
    b^*[x]/f(V_a)=b^*\mathbb{C}\mathrm{P}(V_a)_+
  \end{displaymath}
  which is the identity on representatives is an isomorphism in this degree. 
  In the polynomial ring, the condition~\hbox{$P(m)\equiv h_am$}
  modulo~$r^a$ ensures that the image~$\mu$
  of~$m$
  satisfies~\hbox{$P(\mu)=h_a\mu$}.  
\end{proof}

Taken together, the two previous lemmas imply the following result,
which yields the translation of our problem into polynomial algebra.

\begin{proposition}\label{translation}
  The vector space \hbox{$\Hom{b^*b}{1}{b^*\mathrm{S}^{W_a},b^*\mathbb{C}\mathrm{P}(V_a)_+}$} is isomorphic to the vector space~$M_a$ of
  polynomials~$m$ in the polynomial ring~$\mathbb{F}[\tau,x]$ such that the
  algebraic degree of~$m$ is~$\delta_a$ and~$r^a$ divides~$P(m)-h_am$. An isomorphism is given by associating to a map the representative of the evaluation at the generator.
\end{proposition}

\begin{proof} 
	By Lemma~\ref{lem:1}, evaluation on a generator of~$b^*\mathrm{S}^{W_a}$ gives an isomorphism bet\-ween the vector space $\Hom{b^*b}{1}{b^*\mathrm{S}^{W_a},b^*\mathbb{C}
    \mathrm{P}(V_a)_+}$ and~$\overline{M}_a$. By Lemma~\ref{lem:2}, there is an isomorphism between~$M_a$ and~$\overline{M}_a$ which is the identity on representatives.
\end{proof}


\section{The algebra in the case~$a=2$}\label{sec:a=2}

As a special case of Definition~\ref{def:M},~$M_2\subset\mathbb{F}[\tau,x]$ is the space of polynomials of
degree~\hbox{$3(p-1)/2$} that satisfy~$P(m)\equiv h_2m$ modulo~$r^2$. The purpose of this section is to prove the following estimate.

\begin{proposition}\label{case a=2} 
  The dimension of~$M_2$ is at least~$(p+1)/2$.
\end{proposition}

This will be achieved by first producing enough elements in~$M_2$, and then noting that they are linearly independent.

\begin{lemma} 
	For every integer~$k$ such that~$0\leqslant
    k\leqslant(p-1)/2$, the polynomial
  \begin{displaymath}
    \tau^{\frac{p-1}{2}-k}x^k(kx^{p-1}+(1-k)\tau^{p-1})
  \end{displaymath}
  is in~$M_2$.
\end{lemma}

\begin{proof}
  The cases~$k=1$ and~$k=0$ can easily be dealt with directly. We will
  only deal with the case~$k\geqslant2$.
  
  Let~$m$ be the polynomial displayed in the proposition. Since it has
  the correct degree, it remains to be shown that~$r^2$
  divides~\hbox{$P(m)-h_2m$}. Let us use the
  notation
  \begin{displaymath}
	E_{\tau}=(1+\tau^{p-1})\quad\text{and}\quad E_x=(1+x^{p-1}).
\end{displaymath}
  Then we have~\hbox{$P(\tau)=\tau E_{\tau}$},~\hbox{$P(x)=xE_x$},~\hbox{$r=x(E_x-E_{\tau})$}
  and~\hbox{$h_2=E_{\tau}^{3(p-1)/2}$}. Rearranging the terms, one sees that~\hbox{$P(m)-h_2m$} equals~$\tau^{\frac{p-1}{2}-k}x^k$ times the following term.
  \begin{equation}\label{long term}
    kx^{p-1}(E_x^{p-1+k}E_{\tau}^{\frac{p-1}{2}-k}-E_{\tau}^{3\frac{p-1}{2}})
    +
    (1-k)\tau^{p-1}(E_x^kE_{\tau}^{3\frac{p-1}{2}-k}-E_{\tau}^{3\frac{p-1}{2}})
  \end{equation}
  Since~$k\geqslant2$, the polynomial~$P(m)-h_2m$ is divisible
  by~$x^2$, and so in this case it remains to be shown that the
  term~(\ref{long term}) is divisible by~$(E_x-E_{\tau})^2$. Now
  \begin{eqnarray*}
    E_x^{p-1+k}E_{\tau}^{\frac{p-1}{2}-k}-E_{\tau}^{3\frac{p-1}{2}}
    &=&
    E_{\tau}^{\frac{p-1}{2}-k}(E_x^{p-1+k}-E_{\tau}^{p-1+k})\\
    &=&
    E_{\tau}^{\frac{p-1}{2}-k}(E_x-E_{\tau})\sum_{j=0}^{p+k-2}E_x^{p+k-2-j}E_{\tau}^j
  \end{eqnarray*}
  and 
  \begin{eqnarray*}
    E_x^kE_{\tau}^{3\frac{p-1}{2}-k}-E_{\tau}^{3\frac{p-1}{2}}
    &=&
    E_{\tau}^{3\frac{p-1}{2}-k}(E_x^k-E_{\tau}^k)\\
    &=&
    E_{\tau}^{3\frac{p-1}{2}-k}(E_x-E_{\tau})\sum_{j=0}^{k-1}E_x^{k-1-j}E_{\tau}^j
  \end{eqnarray*}
  are both divisible by~$E_x-E_{\tau}$. It remains to be
  shown that\vspace{2mm}
  \begin{displaymath}
    kx^{p-1}E_{\tau}^{\frac{p-1}{2}-k}(\sum_{j=0}^{p+k-2}E_x^{p+k-2-j}E_{\tau}^j)
    +
    (1-k)\tau^{p-1}E_{\tau}^{3\frac{p-1}{2}-k}(\sum_{j=0}^{k-1}E_x^{k-1-j}E_{\tau}^j)
    \vspace{2mm}
  \end{displaymath}
  is divisible by~$E_x-E_{\tau}$. But modulo~$E_x-E_{\tau}$
  we have~$E_x \equiv E_{\tau}$, so this is
  \begin{eqnarray*}
    &\phantom{=}& kx^{p-1}E_{\tau}^{\frac{p-1}{2}-k}((p+k-1)E_{\tau}^{p+k-2})
    + (1-k)\tau^{p-1}E_{\tau}^{3\frac{p-1}{2}-k}(kE_{\tau}^{k-1})\\
    &=& E_{\tau}^{3\frac{p-1}{2}-1}(k(k-1)x^{p-1}+(1-k)k\tau^{p-1})\\
    &=& E_{\tau}^{3\frac{p-1}{2}-1}(k^2-k)(x^{p-1}-\tau^{p-1}),
  \end{eqnarray*}
  and~$x^{p-1}-\tau^{p-1}=E_x-E_{\tau}\equiv0$ modulo
 ~$E_x-E_{\tau}$. This finishes the proof.
\end{proof}

\begin{proof}[Proof of Proposition~\ref{case a=2}]
The~$(p+1)/2$ elements in the preceding lemma are linear
independent in~$\mathbb{F}[\tau,x]$. This follows from an inspection of the
matrix which expresses these in terms of the monomial basis: that
matrix has full rank.
\end{proof}


\section{The algebra in the cases~$a>2$}\label{sec:a>2}

In this section, we will see that the vector spaces~$M_a$ for~$a>2$ relate to~$M_2$. Before that, some general remarks are in order.

A {\it linear} polynomial in~$\mathbb{F}[\tau,x]$ is one of the
form~$\kappa\tau+\lambda x$ for some~$\kappa$ and~$\lambda$ in~$\mathbb{F}$.
A polynomial will be called \emph{split} if it is a product of linear
polynomials. Assume that~$m$ is split into linear factors of the
form~$L_j=x-\lambda_j\tau$:
\begin{displaymath}
  m=\prod_jL_j.
\end{displaymath}
Since~$P(L_j)=(L_j+L_j^p)=L_j(1+L_j^{p-1})$ and~$P$ is
multiplicative, we have
\begin{displaymath}
  P(m)=m\prod_j(1+L_j^{p-1}).
\end{displaymath}
This proves that, if~$m$ splits as above, then~$m$ divides~$P(m)$. Let us
write
\begin{displaymath}
  Q(m)=P(m)/m
\end{displaymath}
for the quotient in this case, and recall the definitions of the polynomial~$h_a$, and the numbers~$\delta_a$, and~$\epsilon_a$ from~\eqref{eq:h},~\eqref{eq:epsilon}, and~\eqref{eq:delta}. If~$m$ splits, then~$m$
divides also~$P(m)-h_am$, and the quotient is
\begin{equation}\label{subs3}
  Q(m)-h_a=\biggl(\prod_j(1+L_j^{p-1})\biggr)
  -\biggl((1+\tau^{p-1})^{(2a-1)\frac{p-1}{2}}\biggr).
\end{equation}

Let us count the number of factors in both of these terms.  One the
one hand, if~$m$ splits, it does so into~$\delta_a$ factors. On the
other hand, the polynomial~$h_a$ has~$\epsilon_a$
factors~$(1+\tau^{p-1})$. The difference between those two numbers is
\begin{equation}\label{a-2}
  \delta_a-\epsilon_a =a-2,
\end{equation}
so the numbers~$\delta_a$ and~$\epsilon_a$ agree if and only if~$a=2$.

By definition, see (\ref{f}), the element~$r$ in~$\mathbb{F}[\tau,x]$
splits. In order to give a nice formula for~$Q(r)$, some more preliminaries
are necessary.

\vbox{\begin{lemma}
  For any~$\kappa$ in~$\mathbb{F}$, the equation
  \begin{displaymath}
    x^p-\tau^{p-1}x
    =(x-\kappa\tau)(\tau^{p-1}-(x-\kappa\tau)^{p-1})
  \end{displaymath}
  holds in~$\mathbb{F}[\tau,x]$. 
\end{lemma}}

\begin{proof}
  Note that 
  \begin{displaymath}
    \prod_{\lambda\in\mathbb{F}}(x-\lambda\tau)=
    \prod_{\lambda\in\mathbb{F}}(x-(\kappa+\lambda)\tau)=
    \prod_{\lambda\in\mathbb{F}}((x-\kappa\tau)-\lambda\tau),
  \end{displaymath}
  and therefore
  \begin{eqnarray*}
    x^p-\tau^{p-1}x
    &=& (x-\kappa\tau)^p-\tau^{p-1}(x-\kappa\tau)\\
    &=& -(x-\kappa\tau)(\tau^{p-1}-(x-\kappa\tau)^{p-1}),
  \end{eqnarray*}
  which proves the claim.
\end{proof}        

\begin{lemma}
  In~$\mathbb{F}[\tau,x][K]$ the equation
  \begin{displaymath}
    \prod_{\lambda\in\mathbb{F}}(K+(x-\lambda\tau)^{p-1})=
    (x^p-\tau^{p-1}x)^{p-1}+K(K+\tau^{p-1})^{p-1}
  \end{displaymath}
  holds.
\end{lemma}

\begin{proof}
  For any~$\kappa$ in~$\mathbb{F}$ we can substitute
 ~$K=-(x-\kappa\tau)^{p-1}$ into the right hand side of
  the equation. We get
  \begin{displaymath}
    (x^p-\tau^{p-1}x)^{p-1} -(x-\kappa
    \tau)^{p-1}(-(x-\kappa\tau)^{p-1}+\tau^{p-1})^{p-1}.
  \end{displaymath} 
  But this is zero by the equation from the previous lemma, raised to the~$(p-1)$-st power. This argument shows that the
  left hand side of the equation divides the right hand side. The claim
  follows by comparing the degrees and a coefficient.
\end{proof}

One gets the equation
\begin{displaymath}
  \prod_{\lambda\in\mathbb{F}}(1+(x-\lambda\tau)^{p-1})
  =(x^p-\tau^{p-1}x)^{p-1}+(1+\tau^{p-1})^{p-1}
\end{displaymath} 
by specializing to~$K=1$ in the previous lemma. Using different
notation, this equation says the following.

\begin{proposition}
  The equation
  \begin{equation}\label{Q}
    Q(r)=r^{p-1}+(h_{a+1}/h_a)
  \end{equation}
  holds for every~$a \geqslant 2$.
\end{proposition}

Let~$\kappa$ be in~$\mathbb{F}$. Then~$r$ is clearly divisible
by~$x-\kappa\tau$, but~$h_{a+1}/h_a$ is not. Therefore, the equation
(\ref{Q}) above shows that~$Q(r)$ is not divisible
by~\hbox{$x-\kappa\tau$}. 

We can now relate the cases~$a>2$ to the case~$a=2$. Recall that~$M_a$ has been defined to be the set of polynomials~$m$ in~$\mathbb{F}[t,x]$ of degree~$\delta_a$ that satisfy the
condition~\hbox{$P(m)\equiv h_am$} modulo~$r^a$.

\begin{lemma}
  If~$a\geqslant3$ then every element~$m$ in~$M_a$ is divisible by~$r$.
\end{lemma}

\begin{proof}
  Some notation is needed first. We can write
  \begin{displaymath}
    m=\sum_{i+j=\delta_a}c_{i,j}\tau^ix^j
  \end{displaymath}
  with some coefficients~$c_{i,j}$ in~$\mathbb{F}$. Furthermore,
  setting~\hbox{$E_{\tau}=(1+\tau^{p-1})$} and
  similarly~\hbox{$E_x=(1+x^{p-1})$}, we have~\hbox{$P(\tau)=\tau
    E_{\tau}$} as well as~\hbox{$P(x)=xE_x$}. Consequently,
  \begin{displaymath}
    P(m)-h_am=\sum_{i+j=\delta_a}c_{i,j}\tau^ix^j(E_{\tau}^iE_x^j-h_a).
  \end{displaymath}
  By assumption on~$m$, we have~$P(m)-h_am=r^as$
  for some~$s$ in~$\mathbb{F}[\tau,x]$. Putting these
  together gives
  \begin{equation}\label{ausgangspunkt}
    r^as=\sum_{i+j=\delta_a}c_{i,j}\tau^ix^j(E_{\tau}^iE_x^j-h_a).
  \end{equation}
  To prove the claim, it suffices to show that
 ~$x-\kappa\tau$ divides~$m$ for each~$\kappa$
  in~$\mathbb{F}$.
  
  First assume~$\kappa\neq0$. Modulo~$x-\kappa\tau$, we have~$r\equiv0$
  and~$E_x\equiv E_t$, so that equation~(\ref{ausgangspunkt}) above
  shows that
  \begin{displaymath}
    0\equiv\left(\sum_{i+j=\delta_a}c_{i,j}\tau^ix^j\right)(E_{\tau}^{\delta_a}-h_a)
  \end{displaymath}
  modulo~$x-\kappa\tau$. In other words, the right hand side, which is
 ~$m(E_{\tau}^{\delta_a}-h_a)$, is divisible by~$x-\kappa\tau$. Both,
 ~$E_{\tau}^{\delta_a}$ and~$h_a$, are powers of~$1+\tau^{p-1}$. By
  (\ref{a-2}), the exponents differ by~$a-2$, so the assumption on~$a$
  implies that~$E_{\tau}^{\delta_a} \neq h_a$. As~a
  consequence,~$x-\kappa\tau$ does not
  divide~$E_{\tau}^{\delta_a}-h_a$, so it must divide~$m$.
  
  It remains to show that~$x$ divides~$m$. But modulo
 ~$x$, the equation (\ref{ausgangspunkt}) reads
  \begin{displaymath}
    0\equiv c_{\delta_a,0}\tau^{\delta_a}(E_{\tau}^{\delta_a}-h_a).
  \end{displaymath}
  Using~$E_{\tau}^{\delta_a} \neq h_a$ again, it follows
  that~$c_{\delta_a,0}$ is zero. In other words,~$x$
  divides~$m$.
\end{proof}

\begin{lemma}\label{l10}
  If~$3\leqslant a\leqslant p$, division by~$r$ yields an 
  injection from the vector space~$M_a$ into~\hbox{$M_{a-1}$}.
\end{lemma}

\begin{proof}
  The previous result shows that every element of~$M_a$ is divisible
  by~$r$. Given~$rm$ in~$M_a$, one has to show that~$m$ is in
 ~$M_{a-1}$. First of all, the degree of~$m$ is correct. Since~$rm$ is in
 ~$M_a$, it is true that~$P(rm)-h_arm$ is divisible by~$r^a$. By
  assumption,~$r^a$ divides~$r^pmh_{a-1}$, so that~$r^a$ divides
  \begin{eqnarray*} 
    P(rm)-h_arm-r^pmh_{a-1} &=&P(rm)-(h_a+r^{p-1}h_{a-1})rm\\
    &=&Q(r)rP(m)-Q(r)h_{a-1}rm\\ &=&Q(r)(rP(m)-rh_{a-1}m).
  \end{eqnarray*} 
  Since none of the (linear) factors of~$r$ divides~$Q(r)$, we can
  deduce that the polynomial~$rP(m)-rh_{a-1}m$ must be divisible
  by~$r^a$, so that~$r^{a-1}$ divides~\hbox{$P(m)-h_{a-1}m$},
  which shows that~$m$ is in~$M_{a-1}$. Thus, the map is
  well-defined. Since~$r$ is not a zero-divisor, it is injective.
\end{proof}

\begin{lemma}\label{l11}
  Let~$a$ and~$b$ be integers such that~$2\leqslant a\leqslant p-1$
  and~$a\leqslant b$. Then, multiplication with~$r^{b-a}$ is an
  injection from~$M_a$ into the group~$M_b$.
\end{lemma}

\begin{proof} 
  Injectivity is clear, if the map is well-defined at all. But, as
  will be shown now, it is. Equation (\ref{Q}) implies
  that~$Q(r)\equiv h_{a+1}/h_a$ modulo~$r^{p-1}$. Using that
  result~\hbox{$b-a$} times, we get~$Q(r)^{b-a}\equiv h_b/h_a$
  modulo~$r^{p-1}$. Hence we know that
  \begin{displaymath}
    h_aQ(r)^{b-a}mr^{b-a}\equiv h_bmr^{b-a}
  \end{displaymath} 
  modulo~$r^{p-1}r^{b-a}$ and therefore, by the assumption on~$a$,
  also modulo~$r^b$. Suppose now that~$m$ is in~$M_a$. Then~$mr^{b-a}$
  has the required degree. If, furthermore,~$m$ satisfies~$P(m)\equiv
  h_am$ modulo~$r^a$, we have
  \begin{eqnarray*}
    P(mr^{b-a})
    &=&       P(m)r^{b-a}Q(r)^{b-a}\\
    &\equiv&  h_amr^{b-a}Q(r)^{b-a}\\
    &\equiv&  h_bmr^{b-a}
  \end{eqnarray*}
  modulo~$r^b$, and that is what was to be shown.
\end{proof}

The following proposition sums up the preceding three lemmas.

\begin{proposition}\label{shifting}
  Multiplications by~$r$ yield isomorphisms
  \begin{equation*}
	M_2 \cong M_3 \cong\dots\cong M_{p-1} \cong M_p
\end{equation*}
and injections from these into~$M_a$ for every~$a\geqslant2$. 
\end{proposition}


\section{Obstructions}\label{sec:obstructions}

As described in Section~\ref{sec:borel}, the~$p$-power torsion in~$[\mathbb{C}
\mathrm{P}(V_a)_+,\mathrm{S}^{W_a}]^G$ will be detected by the Borel cohomology Adams
spectral sequence
\begin{displaymath}
  \Ext{b^*b}{s,t}{b^*\mathrm{S}^{W_a},b^*\mathbb{C}\mathrm{P}(V_a)_+}
  \Longrightarrow
  ([\mathbb{C}\mathrm{P}(V_a)_+,\mathrm{S}^{W_a}]^G_{t-s})^{\wedge}_p.
\end{displaymath}
The~$E_2$-page of that spectral sequence is hard to come by directly. Instead, we may filter the
projective space~$\mathbb{C}\mathrm{P}(V_a)_+$ by projective
subspaces~$\mathbb{C}\mathrm{P}(V)_+$ for~$V\subseteq V_{a}$. The filtration quotients are linear~$G$-spheres, so we may feed in results from~\cite{Szymik:ESS} one filtration step at a time. In this section, it will be described how this can be done. For the most part, this is straightforward, and only at the end of it, we will spot the crux of the matter. See Lemma~\ref{surjection}, which uses notation introduced at the beginning of Section~\ref{quotients}. The reader may consider to take this for granted for the time being and skip to the next section now to see how it fits into the puzzle.

\subsection{The cases~$V\subseteq V_{a-1}$}\label{first steps}

The following result is a consequence of Proposition~5 of~\cite{Szymik:ESS} and its corollaries.

\vbox{\begin{lemma} Let~$V$ be a complex~$G$-representation,~$L\subseteq
    V$ a complex line, and~$W$ be any~$G$-representation that contains
    \hbox{$\Hom{\mathbb{C}}{}{L,V/L}$} up to isomorphism. Then the
    inclusion of~$\mathbb{C}\mathrm{P}(V/L)_+$ into~$\mathbb{C}\mathrm{P}(V)_+$ induces an
    isomorphism
  \begin{displaymath}
    \Ext{b^*b}{s,t}{b^*\mathrm{S}^W,b^*\mathbb{C}\mathrm{P}(V/L)_+}
    \stackrel{\cong}{\longleftarrow}
    \Ext{b^*b}{s,t}{b^*\mathrm{S}^W,b^*\mathbb{C}\mathrm{P}(V)_+}
  \end{displaymath}
  for~$t-s<\dim_\mathbb{R}{W^G}-\dim_\mathbb{R}{\Hom{\mathbb{C}}{}{L,V/L}^G}$.
\end{lemma}}

An induction on the dimension of~$V$ now proves the following result.

\begin{proposition}
  Let~$V\subseteq V_{a-1}$ be a complex subrepresentation. Then
  \begin{displaymath}
    \Ext{b^*b}{s,t}{b^*\mathrm{S}^{W_a},b^*\mathbb{C}\mathrm{P}(V)_+}=0
  \end{displaymath}
  for~$t-s\leqslant0$. In particular, $\Ext{b^*b}{s,t}{b^*\mathrm{S}^{W_a},b^*\mathbb{C}
  \mathrm{P}(V_{a-1})_+}=0$ for~$t-s\leqslant0$.
\end{proposition}


As a consequence, note~$[\mathbb{C}\mathrm{P}(V_{a-1})_+,\mathrm{S}^{W_a}]^G=0$. Although it will not be necessary to calculate~$[\Sigma \mathbb{C}
\mathrm{P}(V_{a-1})_+,\mathrm{S}^{W_a}]^G$, it will be good to know one of the groups on the
relevant column of the~$E_2$-page. These will be described in the following
proposition.

\begin{proposition}        
  Let~$V$ be a complex subrepresentation of~$V_{a-2}$. Then the vector
  space $\Hom{b^*b}{1}{b^*\mathrm{S}^{W_a},b^*\mathbb{C}\mathrm{P}(V)_+}$ is zero. Let~$U$
  be a complex subrepresentation of~$\mathbb{C} G$, and
  set~\hbox{$V=V_{a-2}\oplus U$}, so that~\hbox{$V_{a-2}\subseteq
    V\subseteq V_{a-1}$}. Then
  \begin{displaymath}
    \dim_\mathbb{F}\Hom{b^*b}{1}{b^*\mathrm{S}^{W_a},b^*\mathbb{C}\mathrm{P}(V)_+}=\dim_\mathbb{C} U.
  \end{displaymath} 
  In particular,~$\dim_\mathbb{F}\Hom{b^*b}{1}{b^*\mathrm{S}^{W_a},b^*\mathbb{C}\mathrm{P}(V_{a-1})_+}=p$.
\end{proposition}

\begin{proof}
  The first part can also be proven by induction on the dimension of~$V$. The result is stated in the form in which it will be used later; but
  the pair~\hbox{$(s,t)=(0,1)$} could be replaced by any pair~$(s,t)$
  such that~\hbox{$t-s\leqslant2$}.
        
  The second part can be proven by induction on~$\dim_\mathbb{C} U$, using Propositions~5 and~9 of~\cite{Szymik:ESS} and their corollaries.
%
\end{proof}

\subsection{The cases~$V_{a-1}\subseteq V$}\label{quotients}

We may now start filtering~$V_a$ beyond~$V_{a-1}$. Given an integer~$\alpha$, let~$\mathbb{C}(\alpha)$ be
the~$G$-representation where a chosen generator of~$G$ acts by
multiplication by~$\exp(2\pi\mathrm{i}\alpha/p)$. For any integer~$k$ such
that~$0\leqslant k\leqslant p$ consider the subrepresentation
\begin{displaymath}
  U_k=\bigoplus_{\alpha=0}^{k-1}\mathbb{C}(\alpha)
\end{displaymath}
of~$\mathbb{C} G$. This gives~a flag
\begin{equation}\label{flag}
  0=U_0\subset U_1\subset\dots\subset U_p=\mathbb{C} G
\end{equation}
with the property that~$U_1$ is the trivial~$G$-line and such that also both~\hbox{$U_{(p+1)/2}/U_1$} and~\hbox{$U_p/U_{(p+1)/2}$}
are isomorphic to~$\mathbb{R} G/\,\mathbb{R}$ as real~$G$-representations.  The
flag~(\ref{flag}) yields a filtration of~$\mathbb{C}\mathrm{P}(V_a)_+$ by
projective spaces~\hbox{$\mathbb{C}\mathrm{P}(V_{a-1}\oplus U_k)_+$}. For~$k=0$
this is~$\mathbb{C}\mathrm{P}(V_{a-1})_+$; for~$k=p$ it is~$\mathbb{C}\mathrm{P}(V_a)_+$.
To compute the relevant part of the~$E_2$-page of the spectral
sequence for~\hbox{$[\mathbb{C}\mathrm{P}(V_{a-1}\oplus U_k)_+,\mathrm{S}^{W_a}]^G_*$}, we
will proceed inductively.  The case~$k=0$ has already been settled in
the previous subsection.  One may therefore assume that~$k\geqslant
1$, and that the~$E_2$-page for~\hbox{$\mathbb{C}\mathrm{P}(V_{a-1}\oplus
  U_{k-1})_+$} has been studied. For typographical reasons, we will use the notation~$V_{a-1,k-1}$ for~\hbox{$V_{a-1}\oplus U_{k-1}$} from now on, and similarly in other cases.
  
Let us first see what the quotient~$G$-spheres of the filtration are. The cofibre of the inclusion of~$\mathbb{C}\mathrm{P}(V_{a-1,k-1})_+$
into~$\mathbb{C}\mathrm{P}(V_{a-1,k})_+$ is given
by~$\mathrm{S}^{\overline{H}(a-1,k)}$ with
\begin{eqnarray*}
  \overline{H}(a-1,k)&=&\Hom{\mathbb{C}}{}{U_k/U_{k-1},V_{a-1,k-1}}\\
    &\cong&V_{a-1}\oplus\Hom{\mathbb{C}}{}{U_k/U_{k-1},U_{k-1}}.
\end{eqnarray*}
The cofibration sequence induces a short exact sequence in Borel
cohomology. Applying the functors $\Ext{b^*b}{s,t}{b^*\mathrm{S}^{W_a},?}$
yields long exact sequences. In order to use them, one will have to
know the groups
\begin{equation}\label{third_term}
        \Ext{b^*b}{s,t}{b^*\mathrm{S}^{W_a},b^*\mathrm{S}^{\overline{H}(a-1,k)}}
\end{equation}
for~$t-s=0,1$. These groups can be simplified as follows.
Since~\hbox{$U_k\subseteq\mathbb{C} G$}, the representation $\Hom{\mathbb{C}}{}{U_k/U_{k-1},U_{k-1}}$ has no trivial summand. If we
write
\begin{displaymath}  
  H(k)=\Hom{\mathbb{C}}{}{U_k/U_{k-1},U_k}\cong\overline{H}(0,k)\oplus\mathbb{C},
\end{displaymath}
then~$H(k)\cong\Hom{\mathbb{C}}{}{U_k/U_{k-1},U_{k-1}}\oplus\,\mathbb{C}$
is a complex~$k$-dimensional representation which has exactly
one trivial summand. Because of
\begin{eqnarray*}
  H(k)\oplus V_{a-1} 
  &\cong&\mathbb{C}\oplus \overline{H}(a-1,k)\\
  \mathbb{R} G\oplus V_{a-1} 
  &\cong&\mathbb{C}\oplus W_a,
\end{eqnarray*}
Proposition~3 from~\cite{Szymik:ESS} implies that the
group~(\ref{third_term}) is isomorphic to
\begin{equation}\label{third_term_after_algebraic_stability}
  \Ext{b^*b}{s,t}{b^*\mathrm{S}^{\mathbb{R} G},b^*\mathrm{S}^{H(k)}}.
\end{equation}
Note that this is independent of~$a$.

The following discussion is divided into three cases: First we deal with the cases where~\hbox{$k\leqslant(p-1)/2$}, then with the case~\hbox{$k=(p+1)/2$}, and finally with the remaining cases where~\hbox{$k\geqslant(p+3)/2$}.

\subsection{The cases~$k\leqslant(p-1)/2$}

Here, $\Ext{b^*b}{s,t}{b^*\mathrm{S}^{\mathbb{R} G},b^*\mathrm{S}^{H(k)}}$ as
in~(\ref{third_term_after_algebraic_stability}) is isomorphic to
$\Ext{b^*b}{s,t}{b^*\mathrm{S}^W,b^*\mathrm{S}^2}$ for some subrepresentation~$W$
of~$\mathbb{R} G$ which properly contains the trivial subrepresentation.

\begin{proposition}\label{dim=p}
  If~$k\leqslant(p-1)/2$ then
  \begin{displaymath}
    \dim_\mathbb{F}{\Ext{b^*b}{s,t}{b^*\mathrm{S}^{W_a},b^*\mathbb{C}\mathrm{P}(V_{a-1,k})_+}}
    =
    \begin{cases}
      \begin{array}{lcl}
        0     & \hspace{5pt} & t-s\leqslant -2\\
        k     & \hspace{5pt} & t-s=-1\\
        0     & \hspace{5pt} & t-s=0\\
        p     & \hspace{5pt} & (s,t)=(0,1)
      \end{array}
    \end{cases}
  \end{displaymath}
  holds. The multiplicative structure is the expected one: In
  the~\hbox{$t-s=-1$} column, multiplication with the generator of
 ~$\mathrm{Ext}^{1,1}$ which represents multiplication with~$p$ is
  injective. In the target, this leads to a free module of rank~$k$
  over the~$p$-adic integers.
\end{proposition}

\begin{proof}
  This is again a straightforward induction on~$k$, using the long exact sequence
  associated with the extension
  \begin{displaymath}
    0\longleftarrow 
    b^*\mathbb{C}\mathrm{P}(V_{a-1,k-1})_+\longleftarrow 
    b^*\mathbb{C}\mathrm{P}(V_{a-1,k})_+\longleftarrow 
    b^*\mathrm{S}^{\overline{H}(a-1,k)}\longleftarrow 
    0,
  \end{displaymath}
  and the data from Figure~8 in Section~4 of~\cite{Szymik:ESS}. 
\end{proof}

\subsection{The first interesting case:~$k=(p+1)/2$}

In this case, the sphere~$\mathrm{S}^{\overline{H}(a-1,\frac{p+1}{2})}$ is the
suspension of~$\mathrm{S}^{W_a}$, so that there is an isomorphism
$\Ext{b^*b}{s,t}{b^*\mathrm{S}^{W_a},b^*\mathrm{S}^{\overline{H}(a-1,\frac{p+1}{2})}}\cong\Ext{b^*b}{s,t}{b^*,b^*\mathrm{S}^1}$.

The long exact sequence~(\ref{typical_les}) shows that \hbox{$\Ext{b^*b}{s,t}{b^*\mathrm{S}^{W_a},b^*\mathbb{C}\mathrm{P}(V_{a-1,(p+1)/2})_+}$} is zero as soon as the condition~$t-s\leqslant-2$ is fulfilled. The next case~\hbox{$t-s=-1$} is easy, too, since then the vector spaces \hbox{$\Ext{b^*b}{s-1,t}{b^*\mathrm{S}^{W_a},b^*\mathbb{C}\mathrm{P}(V_{a-1,(p-1)/2})_+}$} and \hbox{$\Ext{b^*b}{s+1,t}{b^*\mathrm{S}^{W_a},b^*\mathrm{S}^{\overline{H}(a-1,\frac{p+1}{2})}}$} both vanish. In this way we may therefore deduce that the vector space \hbox{$\Ext{b^*b}{s,t}{b^*\mathrm{S}^{W_a},b^*\mathbb{C}\mathrm{P}(V_{a-1,(p+1)/2})_+}$} is an extension of the vector space \hbox{$\Ext{b^*b}{s,t}{b^*\mathrm{S}^{W_a}, b^*\mathbb{C}\mathrm{P}(V_{a-1,(p-1)/2})_+}$} by the vector space \hbox{$\Ext{b^*b}{s+1,t}{b^*\mathrm{S}^{W_a},
    b^*\mathrm{S}^{\overline{H}(a-1,\frac{p+1}{2})}}$}.  Using the data
from~\cite{Szymik:ESS}, Figure~4 in Section~1, one obtains the
following.
\begin{equation}\label{induction}\textstyle
  \dim_\mathbb{F}{\Ext{b^*b}{s,t}{b^*\mathrm{S}^{W_a},b^*\mathbb{C}\mathrm{P}(V_{a-1,(p+1)/2})_+}}=
  \begin{cases}
    \begin{array}{ll}
      (p+1)/2 & s=0\\
      (p+3)/2 & s\geqslant1
    \end{array}
  \end{cases}
\end{equation}
The multiplicative structure is again as expected: In the column
$t-s=-1$, multiplication with the generator of~$\mathrm{Ext}^{1,1}$
which represents multiplication with~$p$ is injective. Starting with~$t-s=0$, the situation becomes more interesting. Then, the map
\begin{displaymath}
  \Ext{b^*b}{s,t}{b^*\mathrm{S}^{W_a},b^*\mathbb{C}\mathrm{P}(V_{a-1,(p+1)/2})_+} 
  \longleftarrow
  \Ext{b^*b}{s,t}{b^*\mathrm{S}^{W_a},b^*\mathrm{S}^{\overline{H}(a-1,\frac{p+1}{2})}}
\end{displaymath}
is surjective, since \hbox{$\Ext{b^*b}{s,t}{b^*\mathrm{S}^{W_a},b^*\mathbb{C}
    \mathrm{P}(V_{a-1,(p-1)/2})_+}$} is zero. As the right hand side is non-zero only for~$(s,t)=(1,1)$, so is \hbox{$\Ext{b^*b}{s,t}{b^*\mathrm{S}^{W_a},b^*\mathbb{C}\mathrm{P}(V_{a-1,(p+1)/2})_+}$}. Thus, it remains to
determine \hbox{$\Ext{b^*b}{1,1}{b^*\mathrm{S}^{W_a},b^*\mathbb{C}\mathrm{P}(V_{a-1,(p+1)/2})_+}$}.

The vector space $\Ext{b^*b}{1,1}{b^*\mathrm{S}^{W_a},
  b^*\mathrm{S}^{\overline{H}(a-1,\frac{p+1}{2})}}$ is 1-dimensional. The
kernel of the surjection displayed right above is isomorphic -- via
the boundary homomorphism of the long exact sequence -- to the
cokernel of the map induced by the homomorphism $\Hom{b^*b}{1}{b^*\mathrm{S}^{W_a},?}$. But that
induced map is injective, since the
group $\Hom{b^*b}{1}{b^*\mathrm{S}^{W_a},b^*\mathrm{S}^{\overline{H}(a-1,\frac{p+1}{2})}}$
is zero. To summarise:

\begin{lemma}\label{labelled_map}
  The vector space $\Ext{b^*b}{1,1}{b^*\mathrm{S}^{W_a},b^*\mathbb{C}\mathrm{P}(V_{a-1,(p+1)/2})_+}$ is either 1-di\-men\-sional or zero, depending on whether the injection
  \begin{displaymath}
\Hom{b^*b}{1}{b^*\mathrm{S}^{W_a},b^*\mathbb{C}\mathrm{P}(V_{a-1,(p+1)/2})_+}\longrightarrow
  \Hom{b^*b}{1}{b^*\mathrm{S}^{W_a},b^*\mathbb{C}\mathrm{P}(V_{a-1,(p-1)/2})_+}
\end{displaymath}
  is also surjective (and therefore an
  isomorphism) or not (in which case the cokernel is 1-di\-mensional).
\end{lemma}

By Proposition~\ref{dim=p}, the space $\Hom{b^*b}{1}{b^*\mathrm{S}^{W_a},b^*\mathbb{C}\mathrm{P}(V_{a-1,(p-1)/2})_+}$ has dimension~$p$. Thus it would be sufficient to know the dimension of the other vector space $\Hom{b^*b}{1}{b^*\mathrm{S}^{W_a},b^*\mathbb{C}\mathrm{P}(V_{a-1,(p+1)/2})_+}$ in order to determine the dimension of the space $\Ext{b^*b}{1,1}{b^*\mathrm{S}^{W_a},b^*\mathbb{C}\mathrm{P}(V_{a-1,(p+1)/2})_+}$. For the moment, let us leave it like that and see how we can proceed.

\subsection{The final cases:~$k\geqslant(p+3)/2$}

In these cases the vector spaces \hbox{$\Ext{b^*b}{s,t}{b^*\mathrm{S}^{\mathbb{R}G},b^*\mathrm{S}^{H(k)}}$} as in~(\ref{third_term_after_algebraic_stability}) are isomorphic to $\Ext{b^*b}{s,t}{b^*,b^*\mathrm{S}^V}$ for some subrepresentation~$V\subseteq\mathbb{R} G$ pro\-perly containing the trivial representation. The calculations summarised in Figure~6 of
Section~3 in~\cite{Szymik:ESS} are relevant here.

To determine the dimension of the vector space $\Ext{b^*b}{s,t}{b^*\mathrm{S}^{W_a},b^*\mathbb{C}\mathrm{P}(V_{a-1,k})_+}$, a long exact sequence~(\ref{typical_les})
will be invoked again. Some of these groups will be non-zero
for~\hbox{$t-s\leqslant-2$}, since this holds
for $\Ext{b^*b}{s+1,t}{b^*\mathrm{S}^{W_a},b^*\mathrm{S}^{\overline{H}(a-1,k)}}$. One
can ignore these, since eventually only the case~$t-s=0$ is of
interest. For the latter, it is good to know about the groups
with~$t-s=-1$. In this case, $\Ext{b^*b}{s-1,t}{b^*\mathrm{S}^{W_a},b^*\mathbb{C}\mathrm{P}(V_{a-1,k-1})_+}$ vanish, and so do the groups
$\Ext{b^*b}{s+1,t}{b^*\mathrm{S}^{W_a},b^*\mathrm{S}^{\overline{H}(a-1,k)}}$. As in the
previous case, we get a splittable short exact sequence. By
induction, using (\ref{induction}),
\begin{displaymath}
  \dim_\mathbb{F}{\Ext{b^*b}{s,t}{b^*\mathrm{S}^{W_a},b^*\mathbb{C}\mathrm{P}(V_{a-1}\oplus
  U_k)_+}}=
  \begin{cases}
    \begin{array}{lcl}
      (p+1)/2  & \hspace{5pt} & s=0\\
      k+1      & \hspace{5pt} & s\geqslant1.
    \end{array}
  \end{cases}
\end{displaymath}
The multiplicative structure is again as expected: In the column
$t-s=-1$, multiplication with the generator of~$\mathrm{Ext}^{1,1}$
which represents multiplication with~$p$ is injective.

Now let us turn to the most interesting situation:~$t-s=0$. If in
addition~\hbox{$s\neq1$}, the
groups $\Ext{b^*b}{s,t}{b^*\mathrm{S}^{W_a},b^*\mathrm{S}^{\overline{H}(a-1,k)}}$ are
zero. Using that and the corresponding result from the first
interesting case as an input, an induction shows that the groups
$\Ext{b^*b}{s,t}{b^*\mathrm{S}^{W_a},b^*\mathbb{C}\mathrm{P}(V_{a-1,k})_+}$ vanish
for~$s\neq1$. For~$k=p$ this implies the following results.

\begin{proposition}\label{prop:elementary abelian}
  All~$p$-power torsion in~$[\mathbb{C}\mathrm{P}(V_a)_+,\mathrm{S}^{W_a}]^G$ has order~$p$. 
\end{proposition}

Therefore, the~$p$-adic completion of that group
is elementary abelian of some rank~$r$. Eventually, we will prove upper and lower bounds
on~$r$, see Section~\ref{sec:conclusion}.

\begin{proposition}
  The non-trivial elements in~$[\mathbb{C}\mathrm{P}(V_a)_+,\mathrm{S}^{W_a}]^G$ are detected by their Borel cohomology e-invariants.
\end{proposition}

It remains to discuss the vector spaces $\Ext{b^*b}{1,1}{b^*\mathrm{S}^{W_a},b^*\mathbb{C}\mathrm{P}(V_{a-1,k})_+}$. This is a bit more complicated than in the
first interesting case since this time the vector space~\hbox{$\Ext{b^*b}{1,1}{b^*\mathrm{S}^{W_a},b^*\mathbb{C}\mathrm{P}(V_{a-1,k-1})_+}$} may already be non-zero. Moreover, the boundary
homomorphism
\begin{displaymath}
  \Ext{b^*b}{2,1}{b^*\mathrm{S}^{W_a},b^*\mathrm{S}^{\overline{H}(a-1,k)}}
  \longleftarrow 
  \Ext{b^*b}{1,1}{b^*\mathrm{S}^{W_a},b^*\mathbb{C}\mathrm{P}(V_{a-1,k-1})_+}
\end{displaymath}
maps into a non-zero group. Nevertheless, it is the zero map. This
follows from the multiplicative structure and the fact that the
boundary map respects this.  Thus, while the argument is a little more
complicated than in the first interesting case, the result is the
same:

\vbox{\begin{lemma}\label{surjection}
  For every~integer~$k$ such that~$(p+1)/2\leqslant k\leqslant p$, the
  map
  \begin{displaymath}
    \Ext{b^*b}{1,1}{b^*\mathrm{S}^{W_a},b^*\mathbb{C}\mathrm{P}(V_{a-1,k-1})_+} \longleftarrow
    \Ext{b^*b}{1,1}{b^*\mathrm{S}^{W_a},b^*\mathbb{C}\mathrm{P}(V_{a-1,k})_+}
  \end{displaymath}
  is a surjection. The kernel of this homomorphism 
  is either 1-di\-men\-sional or
  zero, depending on whether the injection from
  \hbox{$\Hom{b^*b}{1}{b^*\mathrm{S}^{W_a},b^*\mathbb{C}\mathrm{P}(V_{a-1,k})_+}$}
  into \hbox{$\Hom{b^*b}{1}{b^*\mathrm{S}^{W_a},b^*\mathbb{C}\mathrm{P}(V_{a-1,k-1})_+}$} is also surjective~(and therefore an isomorphism)
  or not (in which case the cokernel is 1-dimensional).
\end{lemma}}



\section{Upper and lower bounds}\label{sec:conclusion}

We are now able to put together the algebraic calculations from Sections~\ref{sec:toptoalg},~\ref{sec:a=2},~\ref{sec:a>2}, and the obstruction theory of the previous section to prove our main result, Theorem~\ref{theorem:structure}. The following result summarizes Propositions~\ref{translation},~\ref{case a=2}, and~\ref{shifting} from the former.

\begin{proposition}\label{prop:main calculation} 
  For any integer~$a\geqslant2$, the dimension of the vector space
  \begin{displaymath}
    \Hom{b^*b}{1}{b^*\mathrm{S}^{W_a},b^*\mathbb{C}\mathrm{P}(V_a)_+}
  \end{displaymath}
  is a least~$(p+1)/2$.
\end{proposition}

As for the latter, in Section~\ref{sec:obstructions} we have investigated two
chains
\begin{center}
  \mbox{
    \xymatrix@C=10pt@R=15pt{
      \Ext{b^*b}{1,1}{b^*\mathrm{S}^{W_a},b^*\mathbb{C}\mathrm{P}(V_{a-1,(p-1)/2})_+}
      &
      \Hom{b^*b}{1}{b^*\mathrm{S}^{W_a},b^*\mathbb{C}\mathrm{P}(V_{a-1,(p-1)/2})_+}
      \\
      \Ext{b^*b}{1,1}{b^*\mathrm{S}^{W_a},b^*\mathbb{C}\mathrm{P}(V_{a-1,(p+1)/2})_+}\ar[u]
      &
      \Hom{b^*b}{1}{b^*\mathrm{S}^{W_a},b^*\mathbb{C}\mathrm{P}(V_{a-1,(p+1)/2})_+}
      \ar[u]_{\mathrm{Lemma}~\ref{labelled_map}} \\
      \vdots\rule[-0.5em]{0cm}{1.5em}\ar[u]&\vdots\rule[-0.5em]{0cm}{1.5em}\ar[u]\\
      \Ext{b^*b}{1,1}{b^*\mathrm{S}^{W_a},b^*\mathbb{C}\mathrm{P}(V_{a-1,p})_+}\ar[u]&
      \Hom{b^*b}{1}{b^*\mathrm{S}^{W_a},b^*\mathbb{C}\mathrm{P}(V_{a-1,p})_+}\ar[u]
    }
  }
\end{center}
of homomorphisms. Lemmas~\ref{labelled_map} and~\ref{surjection} show that the homomorphisms on the left hand side are surjective and each kernel is at most 1-dimensional. The maps on the right hand side are injections, and each cokernel is at most 1-dimensional. A map on the left is an isomorphism if and only if the corresponding map on the right is not. Consequently, for any integer~$k$ such that~$(p+1)/2\leqslant k\leqslant p$, the dimensions of $\Ext{b^*b}{1,1}{b^*\mathrm{S}^{W_a},b^*\mathbb{C}\mathrm{P}(V_{a-1,k})_+}$ and \hbox{$\Hom{b^*b}{1}{b^*\mathrm{S}^{W_a},b^*\mathbb{C}\mathrm{P}(V_{a-1,k})_+}$} 
add up to~$p$. Thus, in order to get information on the former, it would be enough to know something about the latter, and this is exactly what Proposition~\ref{prop:main calculation} provides for. This  gives the main result of this text.

\begin{theorem}\label{theorem:structure}
  The~$p$-power torsion of the group~$[\mathbb{C}\mathrm{P}(V_a)_+,\mathrm{S}^{W_a}]^G$ is
  non-zero elementary abelian of rank~$r$ with~$1\leqslant r\leqslant
  (p+1)/2$.
\end{theorem}

\begin{proof} 
  We know from Proposition~\ref{prop:elementary abelian} that the
  group~$[\mathbb{C}\mathrm{P}(V_a)_+,\mathrm{S}^{W_a}]^G$ is elementary abelian.
  
  As for the upper bound on its rank, the length of the chains implies
  that the vector space
  \begin{displaymath}
    \Hom{b^*b}{1}{b^*\mathrm{S}^{W_a},b^*\mathbb{C}\mathrm{P}(V_{a-1,p})_+} =
    \Hom{b^*b}{1}{b^*\mathrm{S}^{W_a},b^*\mathbb{C}\mathrm{P}(V_a)_+}
  \end{displaymath} 
  at the bottom of the right hand side is at
  least~$(p-1)/2$-dimensional, and that
  consequently $\Ext{b^*b}{1,1}{b^*\mathrm{S}^{W_a},b^*\mathbb{C}\mathrm{P}(V_a)_+}$ is at
  most~$(p+1)/2$-dimensional.
  
  In order to obtain the lower bound on the~$p$-torsion of~$[\mathbb{C}
  \mathrm{P}(V_a)_+,\mathrm{S}^{W_a}]^G$, we can use Proposition~\ref{prop:main calculation},
  which states that for any integer~$a\geqslant2$, the dimension of
  the vector space
  \begin{displaymath}
    \Hom{b^*b}{1}{b^*\mathrm{S}^{W_a},b^*\mathbb{C}\mathrm{P}(V_a)_+}
  \end{displaymath}
  is a least~$(p+1)/2$. This implies that one of the inclusions on the
  right chain must be an isomorphism. Therefore, one of the
  surjections on the left cannot be an isomorphism. So the
  group $\Ext{b^*b}{1,1}{b^*\mathrm{S}^{W_a},b^*\mathbb{C}\mathrm{P}(V_a)_+}$ is
  non-zero. By multiplicativity, all the differentials must be
  zero on that group. Thus, the elements survive to
  the~$E_{\infty}$-page.
\end{proof}

\begin{remark}\label{rem:conj}
There are indications that the upper bound given in Theorem~\ref{theorem:structure} is far from being sharp. Using methods similar to those of Section~\ref{sec:a=2}, it can be improved roughly by a factor of~$2$. My results in this direction do not seem to justify a detailed account, in particular, as I suspect that the~$p$-torsion in~$[\mathbb{C}\mathrm{P}(V_a)_+,\mathrm{S}^{W_a}]^G$ is isomorphic to~$\mathbb{Z}/p$ for all primes~\hbox{$p\geqslant3$} and all integers~$a\geqslant2$. While I do not know how to prove this, it is consistent with computer experiments covering all the cases with~$pa\leqslant50$. 
\end{remark}


\section{Blind alleys and dead ends}\label{sec:family}

In this section, we will pursue the question of how the groups~\hbox{$[\mathbb{C}\mathrm{P}(V_a)_+,\mathrm{S}^{W_a}]^G$} may be related for varying~$a\geqslant2$. It follows from Propositions~\ref{translation} and~\ref{shifting} that there are isomorphisms
\begin{eqnarray*}
   \Hom{b^*b}{1}{b^*\mathrm{S}^{W_2},b^*\mathbb{C}
     \mathrm{P}(V_2)_+}
   &\cong&\Hom{b^*b}{1}{b^*\mathrm{S}^{W_3},b^*\mathbb{C}
    \mathrm{P}(V_3)_+}\cong\dots\\
   &\cong&\Hom{b^*b}{1}{b^*\mathrm{S}^{W_p},b^*\mathbb{C}\mathrm{P}(V_p)_+},
\end{eqnarray*}
and injections from these into $\Hom{b^*b}{1}{b^*\mathrm{S}^{W_a},b^*\mathbb{C}\mathrm{P}(V_a)_+}$ for every~$a\geqslant2$. As in the proof of Theorem~\ref{theorem:structure}, this implies that there are~$p$-local isomorphisms
\begin{displaymath} 
    [\mathbb{C}\mathrm{P}(V_2)_+,\mathrm{S}^{W_2}]^G\cong
    [\mathbb{C}\mathrm{P}(V_3)_+,\mathrm{S}^{W_3}]^G\cong\dots\cong
    [\mathbb{C}\mathrm{P}(V_p)_+,\mathrm{S}^{W_p}]^G,
\end{displaymath} 
and injections from these into~$[\mathbb{C}\mathrm{P}(V_a)_+,\mathrm{S}^{W_a}]^G$ for every~$a\geqslant2$. However, the morphisms are given algebraically on the level of Adams spectral sequences by multiplication with a class originating from the regular representation. It would be enlightening to see a more geometric and less computational explanation of the phenomenon. However, there are some blind alleys and dead ends on the way towards such and interpretation, and it seems only fair to disclose three of them here.

\begin{remark}\label{rem:1}
First, note that~$[\mathbb{C}\mathrm{P}(V_a)_+,\mathrm{S}^{W_a}]^G\not\cong[\mathbb{C}
\mathrm{P}(V_{a+1})_+,\mathrm{S}^{W_{a+1}}]^G$ integrally: just compute the
structure of these groups away from~$p$. Thus the phenomenon is genuinely~$p$-local.
\end{remark}

\begin{remark}\label{rem:2}
Second, the groups~$[\mathbb{C}\mathrm{P}(V_a)_+,\mathrm{S}^{W_a}]^G$ for varying~$a$ are
related as shown in the following commutative diagram, in
which all the arrows are induced by inclusions.
\begin{center}
  \mbox{ 
    \xymatrix{
      [\mathbb{C}\mathrm{P}(V_a)_+,\mathrm{S}^{W_a}]^G \ar[r] & 
      [\mathbb{C}\mathrm{P}(V_a)_+,\mathrm{S}^{W_{a+1}}]^G & \\ 
      [\mathbb{C}\mathrm{P}(V_{a+1})_+,\mathrm{S}^{W_a}]^G \ar[u] \ar[r] & 
      [\mathbb{C}\mathrm{P}(V_{a+1})_+,\mathrm{S}^{W_{a+1}}]^G \ar[u] & 
    } 
  }
\end{center}
However it is not possible to explain the phenomenon from this point of view: The
horizontal maps are zero, since they are multiplication with
the Euler class of~$\mathbb{C}G$, which has non-zero fixed
points. The group~$[\mathbb{C}\mathrm{P}(V_{a+1})_+,\mathrm{S}^{W_a}]^G$ seems to
be even more difficult to compute than the other three,
whereas the group~$[\mathbb{C}\mathrm{P}(V_a)_+,\mathrm{S}^{W_{a+1}}]^G$ is zero.
\end{remark}

\begin{remark}\label{rem:3}
Third, one might wonder whether, after translating the situation
into a~\hbox{$(\mathbb{T}\times G)$}-equivariant setting, a suspension
isomorphism could be used to prove Theorem~\ref{theorem:introduction} or
Theorem~\ref{theorem:structure}. But this is
not the case: in the~$(\mathbb{T}\times G)$-equivariant setting,
both sides differ by a~$2p$-dimensional sphere, but~$\mathbb{T}$
acts trivially on one of them and non-trivially on the
other.
\end{remark}


\section*{Acknowledgements}

I thank Stefan Bauer and John Greenlees for discussions and Dan Grayson for help with early experiments using Macaulay2~\cite{M2}.



\vfill

\parbox{\linewidth}{%
Department of Mathematical Sciences\\
NTNU Norwegian University of Science and Technology\\
7491 Trondheim\\
NORWAY\\
\phantom{ }\\
\href{mailto:markus.szymik@math.ntnu.no}{markus.szymik@math.ntnu.no}\\
\href{http://www.math.ntnu.no/~markussz}{www.math.ntnu.no/$\sim$markussz}}


\end{document}